\documentclass{article}
\usepackage{amsmath}
\usepackage[shortlabels]{enumitem}
\usepackage{amssymb}
\usepackage{mathrsfs}
\usepackage{extarrows}
\usepackage{MnSymbol}
\usepackage{tikz}
\usepackage{tikz-cd}
\usetikzlibrary{shapes,positioning,patterns,calc}
\usepackage{amsthm}
\usepackage[backend=bibtex]{biblatex}
\usepackage[boxsize=1.2em, centerboxes]{ytableau}

\def\GG{{\mathbb G}}

\def\NN{{\mathbb N}}

\def\PP{{\mathbb P}}

\def\WW{{\mathbb W}}

\def\ZZ{{\mathbb Z}}

\newcommand{\ul}{\underline}

\def\Gscr{{\mathscr G}}

\def\Nscr{{\mathscr N}}

\def\Pscr{{\mathscr P}}

\def\Xscr{{\mathscr X}}

\def\cal{\mathcal}

\def\cO{{\cal O}}

\def\Spec{\operatorname{Spec}}

\def\Ext{\operatorname{Ext}}

\def\Hom{\operatorname{Hom}}

\def\Proj{\operatorname{Proj}}
\def\Lie{\operatorname{Lie}}

\def\0ol{{\bar 0}}
\def\1ol{{\bar 1}}
\def\2ol{{\bar 2}}
\def\ol2{{\bar 2}}
\def\3ol{{\bar 3}}
\def\4ol{{\bar 4}}
\def\5ol{{\bar 5}}
\def\6ol{{\bar 6}}
\def\7ol{{\bar 7}}
\def\8ol{{\bar 8}}
\def\9ol{{\bar 9}}

\def\P2Skly{\PP^2_{Skly}}

\def\coker{\operatorname {coker}}

\def\Ext{\operatorname {Ext}}

\def\ker{\operatorname {ker}}

\def\Lie{\operatorname {Lie}}

\def\supp{\operatorname {supp}}

\def\Aut{\operatorname{Aut}}

\def\dim{\operatorname{dim}}

\def\Ext{\operatorname{Ext}}

\def\Hom{\operatorname{Hom}}

\def\id{\operatorname{id}}

\def\mod{{\sf mod}\ }

\def\Proj{\operatorname{Proj}}

\def\Rep{{\sf Rep\ }}

\def\Spec{\operatorname{Spec}}

\def\ul1{\operatorname{\underline{1}}}

\def\fR{{\mathfrak R}}

\def\fg{{\mathfrak g}}

\def\fh{{\mathfrak h}}

\def\fn{{\mathfrak n}}
\def\fp{{\mathfrak p}}

\def\ft{{\mathfrak t}}

\def\dirlim{\mathop{\vtop{\baselineskip -100pt\lineskip -1pt\lineskiplimit 0pt
\setbox0\hbox{lim}\copy0\hbox to \wd0{\rightarrowfill}}}\limits}
\def\invlim{\mathop{\vtop{\baselineskip -100pt\lineskip -1pt\lineskiplimit 0pt
\setbox0\hbox{lim}\copy0\hbox to \wd0{\leftarrowfill}}}\limits}

\def\I11{{1 \kern -0.8pt \! \mbox{l}}}
\def\mumu{{\mu\kern-4.2pt\mu}}
\def\bfmu{{\mu\kern-4.2pt\mu}}
\def\2slash{\backslash \! \backslash}

\def\boxtimes{\setbox0\hbox{$\Box$}\copy0\kern-\wd0\hbox{$\times$}}

\def\Supph\operatorname{Supph}

\def\sl{\mathfrak{sl}}

\theoremstyle{definition}

\numberwithin{equation}{subsection}

\newtheorem{thm}[equation]{Theorem}
\newtheorem*{thm*}{Theorem}
\newtheorem{prop}[equation]{Proposition}
\newtheorem{definition}[equation]{Definition}
\newtheorem{example}[equation]{Example}

\newtheorem{lemma}[equation]{Lemma}
\newtheorem*{lemma*}{Lemma}
\newtheorem{cor}[equation]{Corollary}

\newtheorem*{note*}{Note}

\newtheorem*{conjecture*}{Conjecture}
\title{Families of Green rings for abelian restricted Lie algebras}
\author{Justin Bloom}
\date{May 2025}

\begin{document}

\maketitle
\begin{abstract}
    We find equivalent conditions determining the representation type of abelian restricted Lie algebras in terms of how their Green ring of restricted representations varies with respect to different cocommutative Hopf algebra structures on its restricted universal enveloping algebra. Each compatible cocommutative Hopf algebra structure on a tame algebra is shown to have a correspondence between a certain set of Hopf subalgebras, and the set of minimal thick tensor-ideals having identical ring structure when determined by either the Hopf algebra structure or the base
    Lie algebra structure (up to a choice of character group). Those of wild representation type are shown never to have such a correspondence. 
\end{abstract}

\tableofcontents

\section*{Introduction}
The Green ring, or ring of representations, for a finite group scheme, is generally difficult to calculate. For group schemes with infinitely many nonisomorphic indecomposable representations (said to be of infinite representation type), any number of ad-hoc inductive methods may be necessary to calculate the isomorphism type of each tensor pairing. Those of infinite representation type are said to be \textit{tame} if the isomorphism classes of indecomposable representations are parameterized by one continuous and one discrete variable. Any other group scheme of infinite representation type is known to be \textit{wild}, meaning the indecomposable representations are parameterized by uncountably many independent variables (see e.g. Erdmann \cite{Erdmann90}), and hence a complete calculation of the Green ring is untenable.

Our work concerns how the Green ring changes with respect to a Hopf algebra structure on a given algebra. The representations of a finite group scheme $G$ correspond to comodules over the commutative Hopf algebra $\cO(G)$, and to modules over the group algebra $kG := \cO(G)^*.$ The tensor product $V \otimes W$ of $kG$-modules $V$ and $W$ is naturally a $kG \otimes kG$ module, and is given the structure of a $kG$ module via the pullback of comultiplication $kG \to kG \otimes kG.$ That is to say, if we fix a finite associative algebra, the tensor product of modules is dependent on a choice of comultiplication structure. There is much that can be deduced a priori about the Green ring structure, independent of choice of comultiplication structure. For one, comultiplications are assumed cocommutative and coassociative which tells us in turn that the tensor product endows modules with the structure of a symmetric monoidal category, compatible with the forgetful functor to vector spaces. Then the Green ring $\mathfrak{R}$ is a commutative associative ring, with a fixed map $\mathfrak{R} \to \mathbb{Z}$, induced $\ZZ$-linearly from the dimension of modules, and this is a homomorphism of rings no matter which comultiplication structure is chosen to define the product for $\mathfrak{R}$. The more delicate feature of the Green ring we are concerned with is the role of \textit{cohomological support}. 

Support theory offers a topological space $\Xscr(G)$ associated to a finite group scheme $G$, such that each finite representation $V$ has associated to it a closed subspace $\Xscr(G, V)\subset \Xscr(G)$. This assignment satisfies structural properties, including the intersection property
\[\Xscr(G, V \otimes W) = \Xscr( G, V) \cap \Xscr(G, W).\]
Our main Theorem \ref{mainThm} pertains to universal enveloping algebras of abelian restricted Lie algebras as fixed associative algebras of tame and wild representation type. We define in \ref{PAdef} a certain \textit{Property PA} which finds points in the `locus of rigidity' for the Green ring structure with respect to any cocommutative Hopf comultiplication. If two group schemes $G, G'$ correspond to two different cocommutative comultiplications on an algebra $A$ inducing tensor products $\otimes, \otimes'$ of $A$-modules, one wishes to predict when a pair of modules $V, W$ has a (non-canonical, ad-hoc) isomorphism $V \otimes W \cong V \otimes' W$. An algebra has Property PA if a certain prediction in terms of cohomological support holds true. 
Our Theorem \ref{mainThm} states that every abelian restricted enveloping algebra of tame representation type has Property PA, and that every such algebra of wild representation type does not have Property PA. 

Definition \ref{nobleDef} of Property PA is contingent on two fundamental aspects of cohomological support. One aspect comes from the work of Friedlander et al. on the variety of 1-parameter subgroups of an infinitesimal group scheme, which we review in Section \ref{SectionNullcone}. The other aspect is the invariance of support theory with respect to cocommutative Hopf algebra structures on a fixed augmented algebra. The latter follows from a brief argument involving the definition of the cohomological variety as 
\[\Xscr(G) := \Proj H^*(G, k),\]
where $H^*(G, k)= \Ext^*_{kG}(k,k)$, the graded-commutative cohomology ring.
Now, using a restricted Lie algebra $\fg$ to fix an augmented algebra $A = u(\fg)$ with trivial representation $k$ leaves us with support theory $\Xscr(A, V) = \Xscr(\fg, V) = \Xscr(G, V)$ giving the same subspaces of $\Proj \Ext^*_A(k, k)$ for each module $V$ of $A$ and each group scheme structure $G$ compatible with the fixed augmented algebra $A$. The space $\Xscr(A) = \Xscr(A, k)$ consists of 1-parameter subgroups of the infinitesimal group scheme $\widetilde G$ associated to the Lie algebra, but for other group schemes $G$, these subgroups of $\widetilde G$ are visible only as points in the spectrum of cohomology, and not necessarily as subgroups. An algebra has Property PA if, for all possible group algebra structures $kG$ compatible with the augmented algebra $A$ and the character group $\chi(\widetilde G)$, the points of $\Xscr(A)$ which \textit{are} visible as subgroups of $G$ (c.f. \ref{nobleDef}) are all points where the Green ring for $G$ agrees with that of $\fg$.

While the isomorphism type of $V \otimes W$ depends on the comultiplication a priori, the support of $V \otimes W$ is invariant of how $\otimes$ is defined, as it always coincides with 
\[\Xscr( A, V) \cap \Xscr ( A, W).\] For tame algebras $A$ such that modules are classified with respect to the support, this can be a useful invariant for Green ring calculations. For example, the Kronecker algebra $A = k[x,y]/(x^2, y^2)$ in characteric 2 is known to have, for any closed point $\fp \in \Xscr(A)$, only finitely many modules of a given dimension $n$ and support $\{\fp\},$ up to isomorphism. The classification of Hopf algebra structures on the Kronecker algebra, due to X. Wang \cite{XWang13}, is the core of the argument for showing that tame abelian restricted Lie algebras satisfy Property PA in Theorem \ref{mainThm}, which makes use of explicit calculations of tensor products with respect to each member of a complete set of isomorphism classes of (cocommutative) Hopf algebra structures.
Our work on products of modules over the Kronecker algebra in Section \ref{tameSection} intends to illuminate the formulas of Ba{\v s}ev and Conlon, who calculated the Green ring for the Klein 4-group in the 1960s \cite{Basev61}, \cite{Conlon65}. Our main Theorem \ref{mainThm} intends to use the range of isomorphism types of $V \otimes W$ as a measure of complexity for a given augmented algebra. It is shown then for a nice class of algebras that this measure orders the algebras in accordance with representation type. 

\section*{Acknowledgements}
We thank our advisor, Professor Julia Pevtsova, for years of wonderful mentorship. We also thank Professors Dave Benson, Pavel Etingof, James Zhang, and Xingting Wang for lending their expertise in times of need. We thank our colleague Alex Waugh for many lengthy conversations on various connected topics over the years. 
This work has been supported by NSF DMS 220-0832. 
\section*{Preliminaries}
\subsubsection*{Representations of finite group schemes}
\begin{itemize}[label={--}]
    \item We will always use $k$ to denote an algebraically closed field of positive characteristic $p.$ 
    \item \textit{Algebras, coalgebras,} and \textit{bialgebras} over $k$, as well as \textit{modules} and \textit{comodules} over algebras and coalgebras respectively, are as defined in Waterhouse \cite{Wa79}. A \textit{Hopf algebra} is a bialgebra with antipode. 
    \item For a vector space $V$ over $k$, we denote $V^* = \Hom_k(V, k)$ the \textit{linear dual.} The linear dual of an algebra, coalgebra, bialgebra, or Hopf algebra, is naturally a coalgebra, algebra, bialgebra, or Hopf algebra respectively.   
    \item \textit{Group schemes} are always over $k$ and always assumed affine. That is, a group scheme $G$ is defined as $\Spec \cO(G)$ for a commutative Hopf algebra $\cO(G)$, called the coordinate algebra of $G$.
    \item \textit{Finite group schemes} are those $G$ such that $\cO(G)$ is a finite dimensional vector space over $k$. The dimension of $\cO(G)$ is also called the order of $G$, and is denoted $|G|$.
    \item When $G$ is a finite group scheme, we denote by $kG$ the \textit{group algebra}, which is a cocommutative Hopf algebra, given by the linear dual $\cO(G)^*$. As such, $\Spec A^*$ for a finite dimensional cocommutative Hopf algebra $A$ always defines a finite group scheme, and finite group schemes are equivalent to finite dimensional cocommutative Hopf algebras, called group algebras, using the canonical isomorphism $A \xrightarrow{\sim} A^{**}$.
    \item A \textit{grouplike} element $x$ of a Hopf algebra $A$ is an element such that $\Delta(x) = x \otimes x,$ where $\Delta$ is the comultiplication for $A$. The set of grouplike elements in any Hopf algebra forms a group. 
    \item If $G$ is a finite group scheme such that the group algebra $kG$ has a basis of grouplike elements, then $G$ is said to be a \textit{constant group scheme}. Indeed, every finite group is also a discrete topological group with scheme structure the disjoint union of $|G|$ copies of $\Spec k$. This scheme is a constant group scheme and finite groups are equivalent via this construction to constant group schemes over $k$.  
    \item A \textit{representation} of a finite group scheme $G$ is a (left) module over the group algebra $kG.$ Representations are assumed to be finite dimensional as vector spaces over $k$. 
    \item The \textit{representation type} of a finite algebra is said to be \textit{finite, tame,} or \textit{wild}, depending on the difficulty of classifying its modules. Finite, tame, and wild are stated in order of increasing difficulty, with wild thought of as impossible. We will take these notions as defined in Erdmann \cite{Erdmann90}. The representation type of a finite group scheme is defined to be that of its group algebra.
    \item If $G$ is a group scheme, we denote by $\chi(G)$ the \textit{group of characters}, i.e. of isomorphism classes of 1-dimensional representations. The group structure is defined pointwise on homomorphisms $G \to \GG_m,$ and agrees with the group of grouplike elements in the coordinate algebra $\cO(G)$.
    \item A group scheme $G$ is called \textit{unipotent} if there is only one simple representation, up to isomorphism. Then a finite group scheme $G$ is unipotent if and only if $kG$ is a local algebra. 
\end{itemize}
\subsubsection*{Symmetric tensor categories and Green rings}
\begin{itemize}[label={--}]
    \item When $V$ and $W$ are vector spaces over $k$, we denote $V \otimes W$ the \textit{tensor product} of $k$-modules. Then $\otimes$ is a bilinear symmetric monoidal product. 
    \item A \textit{symmetric tensor category} is as defined in Etingov, Gelaki, Nikshych, and Ostrik \cite{EGNO15}. 
    \item For a finite group scheme $G$, we denote $\Rep G$ the category of representations of $G$, i.e. the category of modules over $A = kG$. Then $\Rep G$ is canonically an essentially small symmetric tensor category: 
    \begin{itemize}[label={*}]
        \item If $V$ and $W$ are modules over $A$, then $V \otimes W$ is canonically a module over $A\otimes A$. Then $\otimes$ defines an exact symmetric monoidal product of $A$-modules, with $V \otimes W$ an $A$-module via the pullback along the coassociative cocommutative comultiplication $A \to A \otimes A$.
        \item The counit (augmentation) map $A \to k$ defines $k$ to be a module over $A.$ The module $k$ is the monoidal unit with respect to the product $\otimes.$
        \item If $V$ is an $A$ module, the linear dual $V^*= \Hom_k(V, k)$ is an $A$ module by the action defining a functional
        \[a \cdot f (v) = f(S(a)\cdot v)\]
        for each $f \in V^*$, $a \in A,$ and $v \in V,$ where $S : A \to A$ is the antipode. 
    \end{itemize}
    \item An object $X$ in an abelian category is called \textit{indecomposable} if whenever $X \cong Y \oplus Z$, either $Y \cong {\bf 0}$ or $Z \cong {\bf 0}$. By the Krull-Schmidt theorem, all representations of a group scheme are isomorphic in a unique way to a finite sum of indecomposable representations, up to permutation. This property also follows for symmetric tensor categories, assuming rigidity. 
    \item When $\sf C$ is an essentially small symmetric tensor category, the \textit{Green ring} $\fR({\sf{C}})$ is defined to be the free associative $\mathbb{Z}$-algebra generated by isomorphism classes $[X]$ of objects $X$ in ${\sf C}$, quotiented by the ideal generated by elements in the form 
    \[1 - [{\bf 1}], \qquad [X] + [Y] - [X \oplus Y], \qquad [X][Y] - [X \otimes Y].\]
    Then $\fR({\sf C} )$ is a commutative ring, and for symmetric tensor categories ${\sf C}$,  is isomorphic as a $\ZZ$-module to the free module on isomorphism classes of indecomposable objects of ${\sf C}.$ When $G$ is a finite group scheme, we denote $\fR(G) = \fR(\Rep G)$ the Green ring for $G$. 

\end{itemize}
\subsubsection*{Restricted Lie algebras}
\begin{itemize}[label={--}]
    \item \textit{Restricted Lie algebras} over $k$ are as defined by Jacobson \cite{Jac41}. We assume restricted Lie algebras are always finite dimensional. 
    \item When $\fg$ is a restricted Lie algebra, we adopt the convention that a \textit{Lie subalgebra} $\fh \subset \fg$ is restricted, i.e. a subspace $\fh$ of $\fg$ such that $[\fh, \fh]\subset \fh$ and $\fh^{[p]} \subset \fh.$
    \item When $A$ is a finite associative algebra over $k$, $\Lie^{[p]}(A)$ denotes the restricted Lie algebra on $A$, endowing $A$ with the structures
    \[[x, y ] = xy - yx, \qquad x^{[p]} = x^p.\]
    \item If $\fg$ is a restricted Lie algebra of dimension $r$ over $k$, we denote $u(\fg)$ the \textit{restricted universal enveloping algebra} (or simply the \textit{restricted enveloping algebra}), which is an associative algebra of dimension $p^r$ over $k$. 
    \item Induced maps are defined making $(u, \Lie^{[\fp]})$ an adjoint pair of functors, between the category of restricted Lie algebras and the category of finite associative algebras over $k$. 
    \item \textit{The Poincar{\'e}-Birkhoff-Witt} (PBW) theorem, found in \cite{Jac41}, defines a filtration, called the \textit{PBW filtration}, on $u(\fg)$ such that the associated graded algebra is the quotient $S^*(\fg) / (x^p)_{x \in \fg}$ of the symmetric algebra $S^*(\fg).$
    \item The adjunction map $\fg \to \Lie^{[p]}(u(\fg))$ is injective, so we regard $\fg$ as canonically a subspace of $u(\fg)$. Further, passing to the associated graded is an isomorphism taking $\fg$ to the degree $1$ subspace. 
    \item If $A$ is a coalgebra over $k$ with comultiplication $\Delta : A \to A\otimes A,$ and a fixed compatible unit $1 \in A$ (such that $k \to A$ is a map of coalgebras), denote \[P(A) = \{ x \in A \mid \Delta(x) = x \otimes 1 + 1 \otimes x\}\]
    the \textit{primitive subspace} of $A$. Elements $x \in P(A)$ are called primitive elements. Indeed, $P(A) \subset A$ is a linear subspace. When $A$ is a finite bialgebra, $P(A)$ is a Lie subalgebra of $\Lie^{[p]}(A).$
    \item The map $x \mapsto x \otimes 1 + 1 \otimes x$ is a homomorphism $\fg \to \Lie^{[p]}(u(\fg)\otimes u(\fg)),$ defining a comultiplication $u(\fg) \to u(\fg) \otimes u(\fg)$ making $u(\fg)$ a cocommutative Hopf algebra. 
    It follows from the PBW filtration that $P(u(\fg)) = \fg.$
    \item If $A$ is any Hopf algebra which is generated as an associative algebra by the subspace $\fg = P(A),$ then the map $u(\fg) \to A$, induced from the inclusion $\fg \hookrightarrow \Lie^{[p]}(A)$, is an isomorphism of Hopf algebras.
    \item Let $\fg$ be a restricted Lie algebra. Since $u(\fg)$ is a cocommutative Hopf algebra, we have $\widetilde G = \Spec u(\fg)^*$ is a finite group scheme. We call $\widetilde G$ the \textit{infinitesimal group scheme} associated to $\fg$. We will refrain from referencing infinitesimal group schemes in general, which are defined to be finite and topologically connected. 
    \item A \textit{representation} of a restricted Lie algebra $\fg$ is a representation of its infinitesimal group scheme $\widetilde G$, i.e. a module over $u(\fg)$ which is finite dimensional as a vector space over $k$. Elsewhere in the literature, what we call representations are often distinguished as `restricted representations', but we refrain from referencing Lie algebras in general to warrant this distinction. We denote $\Rep \fg = \Rep \widetilde G$ the category of representations, a symmetric tensor category. The representation type of a restricted Lie algebra is that of its restricted enveloping algebra. 
    \item For $\fg$ a restricted Lie algebra, we denote $\fR(\fg) = \fR(\Rep \fg) = \fR( \widetilde G)$ the Green ring for $\fg.$ 
\end{itemize}

\section{Abelian restricted Lie algebras}\label{abelianLieAlg}
We begin by reviewing the geometry of the nullcone for restricted Lie algebras, as well as the machinery of $\pi$-points due to Friedlander and Pevtsova \cite{FrPev07}, \cite{FrPev05}. We include in Section \ref{SeligmanSection} the classification of abelian restricted Lie algebras due to Seligman \cite{Sel67}, according to representation type. We conclude this section with definitions of noble points and of Property PA, and with the statement of our main Theorem \ref{mainThm}. 
\subsection{Nullcone and cohomology}\label{SectionNullcone}
\begin{definition}
Let $\fg$ be a restricted Lie algebra over a field $k$ of characteristic $p$. The \textit{$r$th restricted nullcone} of $\fg$, denoted $\Nscr_{r}(\fg)$, is defined to be the subset of $\fg$ given by 
\[\Nscr_{r}(\fg) = \{ x \in \fg \mid x^{[p]^r} = 0\}.\]
With $\Nscr_0(\fg) = {\bf 0}$ we also define $\Nscr(\fg) = \bigcup_{r}\Nscr_{r}(\fg)$ to be the \textit{nullcone} of $\fg.$
\end{definition}
Given a restricted Lie algebra $\fg$, the elements in the difference of subsets $\Nscr_r(\fg) \setminus \Nscr_{r-1}(\fg)$ are said to have \textit{order $r$}. When $x \in \Nscr(\fg)$ has order $r$, denote $\langle x \rangle \subset \fg$ the Lie subalgebra with basis $x, x^{[p]}, \dots, x^{[p]^{r-1}}.$ 
Each $\Nscr_r(\fg)$ is a homogeneous closed subvariety (i.e. a cone) of the affine space with points in $\fg$. The work of Suslin, Friedlander, and Bendel \cite{SFB97}
establishes a homogeneous homeomorphism 
\[\fp : \Nscr_1(\fg) \setminus {\bf 0} \to \Spec H^*(\fg, k)\]
of closed points,
hence we also have $\PP(\Nscr_1(\fg)) \to \Proj H^*(\fg, k)$ a homeomorphism of closed points. Here we have abused notation to construct a scheme $\Proj H^*(\fg, k)$, which we define actually as $\Proj R$, where $R$ is the commutative subring of $H^*(\fg, k)$ defined by
\[R = \begin{cases}\sum_{i \ge 0} H^{2i}(\fg, k) & p> 2\\
H^*(\fg, k) & p = 2.\end{cases}\]
In fact the ring $H^*(\fg, k)$ is a finitely generated algebra, hence so is $R$, as shown for restricted Lie algebras by Friedlander and Parshall in \cite{FrPar86annalen}. Later Friedlander and Suslin \cite{FS97} would generalize this result, as well as the results of Golod, Venkov, and Evens on finite groups \cite{Venkov59}, \cite{Golod59}, \cite{Evens61}, to cohomology rings for arbitrary finite group schemes. Thus, we will refer to $\Proj H^*(G, k)$ as a projective variety, possibly reducible, and forgetting any structure sheaf and nilpotency in favor of homogeneous coordinate systems when necessary, whenever $G$ is a finite group scheme (by the same abuse of notation on the graded commutative cohomology ring). Similarly we have support spaces for finite representations of $G$, to be defined later, which we will refer to also as projective varieties.

Given a nonzero $x \in \Nscr_1(\fg)$, the point $\fp(x) \in \Proj H^*(\fg, k)$ is defined to be the radical of the kernel of the map $H^*(\fg, k) \to H^*(\langle x \rangle ,k)$
induced by the inclusion map $\langle x \rangle \hookrightarrow \fg.$ It is shown that $\fp$ is well defined, i.e. that the induced map of graded rings is always nondegenerate (c.f. \textit{generalized $\pi$-points} below).  
\begin{definition}
    Let $G$ be a finite group scheme over an algebraically closed field $k$ of characteristic $p$. 
    \begin{itemize}[label={--}]
        \item A \textit{$\pi$-point} (over $k$) of $G$ is a flat map $\alpha : k[t]/t^p \to kG$ of associative algebras which factors through the inclusion $kU \hookrightarrow kG$ of some unipotent abelian subgroup scheme $U < G.$
        \item Two $\pi$-points $\alpha, \beta$ are said to be equivalent, denoted $\alpha \sim \beta$, if for any finite module $M$ over $kG$, the pullback $\alpha^*(M)$ is free over $k[t]/t^p$ if and only if $\beta^*(M)$ is free. 
        \item If $\alpha$ is a $\pi$-point for $G$, the projective point $\fp(\alpha) \in \Proj H^*(G, k)$ induced by $\alpha$ is defined to be the radical of the kernel of the induced map
        \[H^*(G, k) \xrightarrow{H^*(\alpha)} H^*(k[t]/t^p, k),\]
        of augmented algebra cohomology rings. 
    \end{itemize}
\end{definition}
Friedlander and Pevtsova originally defined flat maps factoring through an abelian subgroup to be `abelian $p$-points' in \cite{FrPev05}. In \cite{FrPev07}, $\pi$-points were defined over field extensions of $k$, requiring a unipotent abelian factor, both to give support spaces that include non-closed points for infinite modules, and to make corrections for proving geometric properties of support varieties for finite group schemes.
We will continue to assume all $\pi$-points are defined over $k$ while taking the results of \cite{FrPev07} as given. 

It is shown that $\fp(\alpha)$ is always a closed point in $\Proj H^*(G, k)$, i.e. that the induced map $H^*(\alpha)$ of graded rings is nondegenerate. Further, it is shown for each closed point $\fp \in \Proj H^*(G, k)$, that there exists a $\pi$-point $\alpha$ such that $\fp = \fp(\alpha)$, and that two $\pi$-points $\alpha, \beta$ are equivalent if and only if $\fp(\alpha) = \fp(\beta).$

Now we see that for a restricted Lie algebra $\fg$, given an element $ x \in \Nscr_1(\fg) \setminus {\bf 0}$, we have the inclusion $\alpha : u(\langle x \rangle ) \to u(\fg)$ is a $\pi$-point of the infinitesimal group scheme $\widetilde G$. We have also that $\fp(x) = \fp(\alpha)$ by definition. For the infinitesimal group scheme associated to some $\fg$, it follows that each $\pi$-point is equivalent to the inclusion $u(\langle x \rangle ) \to u(\fg)$ for some $x \in \Nscr_1(\fg) \setminus {\bf 0}.$

\begin{definition}
    Let $G$ be a finite group scheme, and let $M$ be a representation of $G$. The \textit{cohomological support} of $M$, denoted $\supp_G(M)$, is defined by the subvariety of $\Proj H^*(G, k)$ cut out by the ideal given as the kernel
    \[\Ext_{kG}^*(k, k) \xrightarrow{-\otimes M} \Ext_{kG}^*(M, M).\]
\end{definition}
It is shown in \cite{FrPev07}, \cite{FrPev05} that the closed points of $\supp_G(M)$ are equivalent, under the correspondence with equivalence classes of $\pi$-points, to the set
\[\{[\alpha] \mid \alpha^*(M)\text{ is not free over }k[t]/t^p\}.\]
The latter is called \textit{$\pi$-support}, or rather the closed points of $\pi$-support as originally defined. 
We denote the set of closed points in $\Proj H^*(G, k)$ by $\Xscr(G)$, and 
\[\Xscr(G, M) = \supp_G(M) \cap \Xscr(G),\]
identified with the $\pi$-support of $M$ as a set of points over $k$.

In light of our study comparing cocommutative Hopf algebra structures on a given algebra, we say an augmented algebra $A$ is \textit{groupable} if it isomorphic to the group algebra for some finite group scheme $G$, augmented by its counit. 
Given two cocommutative comultiplications $\Delta, \Delta' : A  \to A \otimes A$ for group scheme structures $G, G',$ giving products $\otimes, \otimes'$, we have for any finite module $M$ that 
\[\supp_G(M) = \supp_{G'}(M).\]
This can be argued by relating both back to $\pi$-support:

A $\pi$-point $\alpha : k[t]/t^p \to A$ for $G$ is not necessarily a $\pi$-point for $G'$ by definition, as the existence of a unipotent subgroup depends on the group structure. However, the induced map $H^*(\alpha)$ is nondegenerate, i.e. $H^+(A, k) \not\subseteq \sqrt{\ker H^*(\alpha)}$, provided any group structure exists such that $\alpha$ factors through a unipotent subgroup algebra. Such nondegenerate flat maps $\alpha$ may be called \textit{generalized $\pi$-points}, and by definition $\fp(\alpha) = \sqrt{\ker H^*(\alpha)}$ is a closed point in $\Proj H^*(A, k)$, since the reduced algebra $H^*(k[t]/t^p, k)_{\mathrm{red}}$ is isomorphic to the graded algebra $k[\xi]$ for some positive degree generator $\xi$. The same equivalence relation $\sim$ is defined for generalized $\pi$-points $\alpha, \beta$ as for $\pi$-points, and $\alpha \sim \beta$ if and only if $\fp(\alpha) = \fp(\beta)$. Hence when $A$ is a groupable augmented algebra, we can define $\supp_A(M)$ unambiguously as a set of equivalence classes of generalized $\pi$-points, realized as the subset of $\Proj H^*(A, k ),$ and coinciding with $\supp_G(M)$ for any group scheme structure $G$ on $A$. Likewise we denote unambiguously $\Xscr(A, M) \subset \Xscr(A)$ the sets of closed points for a groupable augmented algebra $A$ and any module $M$. 



\subsection{Seligman's structure theorem}\label{SeligmanSection}
Let $k$ be an algebraically closed field of characteristic $p$. We denote by $\fn_n$ the $p$-nilpotent cyclic Lie algebra of dimension $n$, i.e.
\[\fn_n = \langle x_1, \dots, x_n \mid x_i^{[p]} = x_{i+1}, \text{ where }x_{n+1} = 0\rangle.\]
We also denote by $\ft = \langle x  \mid x^{[p]} = x\rangle$ the $1$-torus.

We present a theorem of Seligman to give a complete description \ref{RepTypeClassification} of abelian restricted Lie algebras according to representation type. 
\begin{thm}\label{SeligmanStructure}(Seligman, 1967 \cite{Sel67}) Let $\fg$ be an abelian restricted Lie algebra of finite dimension over $k$. Then $\fg$ has a direct sum decomposition as 
\[\fg \cong \ft^r \oplus \sum_{i \ge 1} \fn_i^{s_i}\]
for some $r \ge 0$ and finitely many nonzero $s_i \ge 0.$
\end{thm}
The restricted Lie algebra $\ft^r$ is semisimple. To see this, we have \[u(\ft) \cong k[x]/(x^p - x) \cong \prod_{i = 1}^p k\]
as associative algebras, hence $\ft$ is semisimple.
For restricted Lie algebras $\fg_1, \fg_2,$ we also have $u(\fg_1\oplus \fg_2) \cong u(\fg_1)\otimes u(\fg_2)$ as Hopf algebras. Therefore $\fg\oplus \ft$ has the same representation type (semisimple, finite, tame, wild) as $\fg$, for any restricted Lie algebra $\fg$, since $u(\ft \oplus \fg) \cong \prod_{i =1 }^p u(\fg)$ as associative algebras. 

For an abelian restricted Lie algebra $\fg$, it follows from definitions that the nullcone $\Nscr(\fg) \subset \fg$ is actually a Lie subalgebra of $\fg$ (for nonabelian $\fg$, we may have $\Nscr(\fg)$ is not even a linear subspace). By Theorem \ref{SeligmanStructure}, we have that $\Nscr(\fg)$ is the maximal unipotent Lie subalgebra of $\fg$, and the representation type of $\fg$ agrees with that of $\Nscr(\fg).$ In fact, $u\left(\Nscr(\fg)\right)$ is the principal block of $u(\fg)$ (see e.g. Erdmann \cite{Erdmann90}). 

\begin{note*} Let $\fg$ be an abelian restricted Lie algebra, say $\fg \cong \ft^r \oplus \sum_{i \ge 1} \fn_i^{s_i}.$ It follows, for $\widetilde G$ the associated infinitesimal group scheme, that $\chi(\widetilde G) \cong (\ZZ / p)^r.$ 
\end{note*}
The following theorem and its corollary, classifying abelian restricted Lie algebras according to representation type, appeared in the unpublished \cite{Bloom24}, and is a straightforward consequence of Seligman's structure theorem. 
\begin{thm}\label{RepTypeClassification}
    Let $\fg$ be an abelian restricted Lie algebra of finite dimension over $k$, and let $n$ be the dimension of $\Nscr(\fg)$. 
    \begin{enumerate}[I.]
        \item If $\Nscr( \fg)$ is cyclic (i.e. isomorphic to $\fn_n$), then $\fg$ is of finite representation type. 
        \item If $\Nscr(\fg)$ is not cyclic and $p^n = 4$ (i.e. $p = n = 2$), then $\fg$ is of tame representation type. 
        \item In any other case $p^n > 4$ and we have $\fg$ is of wild representation type. 
    \end{enumerate}
\end{thm}

\begin{cor}\label{primordial}
Let $\fg$ be an abelian restricted Lie algebra of wild representation type with no nontrivial wild direct summands (for any decomposition $\fg \cong \fg'\oplus \fg'',$ if $\fg'$ is of wild representation type then $\fg'' = 0$).
\begin{enumerate}[I.]
    \item If $p = 2,$ then $\fg = \fn_1 \oplus \fn_1 \oplus \fn_1,$ or $\fg = \fn_n \oplus \fn_m$ for $n + m \ge 3,$ and $n, m \ge 1$.
    \item If $p > 2,$ then $\fg = \fn_n \oplus \fn_m$ for $n, m \ge 1$.
\end{enumerate}
\end{cor}

\subsection{Noble points and Property PA}
We finish with the statement of our main theorem. First we define `noble points', as a distinguished class of points in the spectrum $\Proj H^*(G, k)$ which depends on the group scheme structure $G$. Thus for a fixed finite augmented algebra $A$, two different group scheme structures $G, G'$ on $A$ may produce in general different sets of noble points in $\Proj H^*(A, k).$
\begin{definition}\label{nobleDef}
    Let $G$ be a finite group scheme. A point $\fp \in \Proj H^*(G, k)$ is \textit{noble for $G$} if there exists a subgroup $C < G$ such that $kC \cong k[t]/t^p$ and the inclusion $\alpha : kC \hookrightarrow kG$ is a $\pi$-point for which $\fp = \fp(\alpha).$
\end{definition}

Note that in Definition \ref{nobleDef}, for the subgroup $C < G$ we have either $C \cong \ZZ / p$ or $C \cong \GG_{a(1)},$ by a theorem of Tate and Oort \cite{OT1970}, which we include as Theorem \ref{TOClassification1}. 
Next we define noble correspondence. A noble correspondence between group scheme structures on a fixed augmented algebra is meant to witness how restriction along a subgroup isomorphic to one of $\GG_{a(1)}$ or $\ZZ / p$ may help determine the isomorphism type of $M\otimes N$ for modules $M, N$ having support $\{\fp\}$ a noble point. 
\begin{definition}
    Suppose $\Delta, \Delta' : A \to A \otimes A$ are two cocommutative comultiplications on a finite augmented algebra $A$, for two group scheme structures $G, G'$. We say that $G$ and $G'$ are in \textit{noble correspondence} if, for each closed point $\fp \in \Proj H^*(A, k)$ which is noble for $G$ and $G'$, that
    \[M \otimes N \cong M \otimes ' N\]
    for each pair of finite modules $M, N$ having support $\{\fp\},$ where $\otimes, \otimes'$ are the two tensor products of modules defined by $\Delta, \Delta'$. 
\end{definition}

Now Property PA demands that a base group scheme structure $\widetilde G$ admits a noble correspondence with any other structure sharing its group algebra as an augmented algebra with fixed character group. 
\begin{definition}
    Let $\widetilde G$ be a finite abelian unipotent group scheme such that every closed point in $\Proj H^*(\widetilde G, k)$ is noble, and let $A = k\widetilde G$ be the group algebra. We say $\widetilde G$ has \textit{Property PA} if, for any cocommutative comultiplication $\Delta : A \to A \otimes A$ for a group scheme $G$, we have that $\widetilde G$ and $G$ are in noble correspondence. 
\end{definition}
\begin{definition}\label{PAdef}
    Let $\widetilde G$ be a finite group scheme such that every closed point in $\Proj H^*(\widetilde G, k)$ is noble, and let $A = k\widetilde G$ be the group algebra. We say $\widetilde G$ has \textit{Property PA} if, for any cocommutative comultiplication $\Delta : A \to A \otimes A$ for a group scheme $G$ with $\chi( G) = \chi(\widetilde G)$, we have $\widetilde G$ and $G$ are in noble correspondence. 
\end{definition}
It follows from definitions that Property PA is unambiguous for finite abelian unipotent group schemes.
\begin{thm}\label{mainThm}
    Let $\widetilde G$ be the infinitesimal group scheme associated to some abelian restricted Lie algebra. 
\begin{itemize}[label={--}]
    \item If $\widetilde G$ is of tame representation type, then $\widetilde G$ has Property PA. 
    
    \item If $\widetilde G$ is of wild representation type, then $\widetilde G$ does not have Property PA. 
    
    \item If $\widetilde G$ is of finite representation type and the principal block $B$ of $k \widetilde G$ has dimension $|B| \le p^3$ then $\widetilde G$ has Property PA.
\end{itemize}
\end{thm}

We conjecture further that, for any such $\widetilde G$, having Property PA is equivalent to having tame or finite representation type. However, we find that classifying the different group algebra structures for the augmented $k\widetilde G$ becomes too difficult to affirm Property PA, when the principal block has large order. Some simple observations on a certain finite family of nonisomorphic cocommutative Hopf algebra structures on $k[t]/t^{p^n}$ finds noble correspondences consistent with Property PA, but it is not known if this is sufficient. 

\section{Hopf algebra structures}\label{HopfStructures}
For a given finite algebra $A$, it is in general a difficult problem to classify the Hopf algebra structures on $A$ up to isomorphism. We will make use of the work of X. Wang \cite{XWang13}, \cite{XWang15}, which classifies \textit{connected} Hopf algebras of small dimension, and meeting certain constraints relevant for our unipotent abelian group schemes. Dualizing such Hopf algebras gives local algebras, and in particular all Hopf algebra structures on any given local algebra of dimension $p^2$ over an algebraically closed field of characteristic $p$ are classified. Our main Theorem \ref{mainThm} only affirms Property PA in the case where the principal block of an algebra is of dimension $\le p^3$ because Property PA is quantified over all possible group scheme structures on a given algebra, and as such we limit ourselves to the scope of X. Wang's classification. 
Negating Property PA for a given algebra requires only constructing one pair of group scheme structures not in noble correspondence, and may therefore be proven independent of a classification of all group scheme structures. 

Let's assume throughout that $A$ has a fixed augmentation $\varepsilon : A \to k$. There is a certain subvariety $\Gscr_A$ of points \[(\Delta, S) \in \Hom_k(A, A \otimes A) \times \Hom_k(A, A)\] such that $\Delta$ is a coassociative, cocommutative map of algebras with antipode $S$ and counit $\varepsilon$. We call $\Gscr_A$ the variety of \textit{group scheme structures} on $A$. By Tannakian duality, $\Gscr_A$ is equivalent to the variety of tensor category structures on $\mod A$ with a fixed monoidal unit $k$, such that the forgetful functor to vector spaces is a tensor functor (see e.g. Etingof, Gelaki, Nikshych and Ostrik \cite{EGNO15}). This motivates our study of how the Green ring $\fR(G)$ changes for group scheme structures $G \in \Gscr_A$ on a fixed $A$.

For an augmented algebra $A$, we denote by $\Aut(A)$ the group of augmented automorphisms. The group $\Aut(A)$ acts on the variety $\Gscr_A,$ and hence on the set of Green ring structures on the free abelian group generated by finite indecomposable modules. The action of $\varphi \in \Aut(A)$ takes $G = (\Delta, S) \in \Gscr_A$ to $G^{\varphi} = (\Delta^{\varphi}, S^{\varphi})$, where 
\[\Delta^{\varphi} = (\varphi\otimes \varphi) \circ \Delta \circ \varphi^{-1},\qquad S^{\varphi} = \varphi \circ S \circ \varphi^{-1}. \]
When $\otimes$ is the tensor product for $G$, defined by pulling back along $\Delta$, we denote also $\otimes^{\varphi}$ the tensor product for $G^{\varphi}$. If $M$ is a (left) module over $A$ and $\varphi \in \Aut(A),$ we denote by $M^{\varphi}$ the base change $\varphi_!(M)$, so that if $m : A \otimes M \to M$ is the module structure on $M$, we have $m^{\varphi} : A \otimes M \to M$ is given by \[m^{\varphi} = m \circ (\varphi^{-1} \otimes \id_M).\]
Now we have the natural identity 
\begin{equation}\label{phiDist}(M \otimes N )^{\varphi} = M^{\varphi} \otimes^{\varphi} N^{\varphi},
\end{equation}
for any modules $M, N$, for any $G$ with product $\otimes$, and any $\varphi$. 

\begin{note*}
     Our Definition \ref{PAdef} of Property PA is ranging over comultiplications $A \to A \otimes A$ `coming from' a group scheme $G$. This may be rephrased by ranging over points $G$ in our variety $\Gscr_A$ of group scheme structures on $A$. 
\end{note*}
\subsection{Unipotent abelian group schemes of order $p^n$}\label{UniAbSection}
Let $G$ be a unipotent group scheme of order $p^n,$ i.e. a finite group scheme with group algebra $A = kG$, a local algebra of dimension $p^n$. Then the dual Hopf algebra $A^*$ is a commutative algebra, and is \textit{connected}, i.e. has a coradical of dimension $1$. If $n \le 2$, then $A^*$ is necessarily isomorphic to one of the connected Hopf algebras in X. Wang's classification \cite{XWang13}. 
In fact, there are finitely many isomorphism classes given, and a complete classification of local Hopf algebras of dimension $p^2$ are given as an explicit corollary in \cite{XWang13}, meaning $A$ is isomorphic to one appearing in the paper as well. 
If $n = 3$, $A^*$ is necessarily isomorphic to one of the connected Hopf algebras given by the classification Nguyen, L. Wang, and X. Wang \cite{NWW15}, or the unofficial sequel \cite{XWang15} of X. Wang classifying those of large abelian primitive subspaces. This classification involves an infinite parameterized family of nonisomorphic Hopf algebras. 
Nevertheless we can still identify our variety $\Gscr_A$ for a fixed augmented algebra $A$ as isomorphic to some closed fiber within the spaces of connected Hopf algebras relevant to \cite{NWW15}, \cite{XWang13}, \cite{XWang15}. 
For each of the group schemes $\widetilde G$ that we affirm Property PA for, we find there are finitely many orbits in $\Gscr_A$ (where $A = k\widetilde G$ is the augmented algebra). Such a result may be open for larger groups of finite representation type, although preliminary computations affirm that a natural interpolation of the classifications listed below seems to hold in general.

\begin{definition}
    Let $A$ be a commutative, cocommutative Hopf algebra, with $G = \Spec A^*$ the finite abelian group scheme with $A = kG$ and $A^* = \cO(G).$ Then we denote $G^* = \Spec A$ the \textit{Cartier dual} of $G$, for which $kG^* = A^*$ and $\cO(G^*) = A.$
\end{definition}

\begin{definition}\label{WittVectors}
    Let $\fg = \fn_n$ be the cyclic Lie algebra of dimension $n$, as in \ref{SeligmanStructure}. The infinitesimal group scheme $\widetilde G$ associated to $\fg$ is denoted $\WW_{n(1)}.$ This notation comes from $\WW_{n(1)}$ being isomorphic to the first Frobenius kernel of $\WW_n,$ the additive group of \textit{length $n$ Witt vectors}. See e.g. Demazure and Gabriel \cite[\!\!\!~V]{DemGab70} for background.
    The Cartier dual $\WW_{n(1)}^*$ is isomorphic to $\WW_{1(n)} \cong \GG_{a(n)},$ the $n^{th}$ Frobenius kernel of $\WW_1 \cong \GG_a,$ the general additive group scheme.  
\end{definition}
The formulas in the following classifications will make frequent use of an expression $\omega(t) \in A\otimes A,$ defined whenever $t \in A$, with $A$ an algebra over $k$ of characteristic $p$. We denote
\[\omega(t) = \frac{(t\otimes 1 + 1 \otimes t)^p -(t^p \otimes 1 + 1 \otimes t^p)}{p} = \sum_{i = 1}^{p-1} \frac{(p-1)!}{i!(p-i)!}t^i\otimes t^{p-i} .\]

\begin{thm}\label{TOClassification1}(Tate and Oort, 1970 \cite{OT1970}) Let $\fg =\fn_1$, the $p$-nilpotent cyclic Lie algebra, and let $A = u(\fg) = k[x]/x^p.$ Then a complete set of representatives for the orbit space $\Gscr_A / \Aut(A)$ is given as follows:
\begin{enumerate}\setcounter{enumi}{-1}
\item $\widetilde G = \GG_{a(1)},$ with
\[\widetilde \Delta : x \mapsto x \otimes 1 + 1 \otimes x,\]
\item $G_1 = \ZZ / p,$ with
\[\Delta_1 : x \mapsto x \otimes 1 + 1 \otimes x + x\otimes x.\]
\end{enumerate}
\end{thm}

\begin{thm}\label{WangClassification2}(X. Wang, 2013 \cite{XWang13})
    Let $\fg$ be a 2-dimensional unipotent abelian restricted Lie algebra and let $A = u(\fg).$

    If $\fg = \fn_2,$ and $A = k[x]/(x^{p^2})$, then a complete set of representatives for the orbit space $\Gscr_A / \Aut(A)$ is given as follows:
    \begin{enumerate}\setcounter{enumi}{-1}
        \item $\widetilde G = \WW_{2(1)}$, with \[\widetilde \Delta : x \mapsto  x\otimes 1 + 1 \otimes x,\]
        \item $G_1$, the Cartier dual of a certain subgroup of $\WW_{2}$, with \[\Delta_1 : x \mapsto x \otimes 1 + 1 \otimes x + \omega(x^p),\]
        \item $G_2 = \ZZ / (p^2),$ with\[ \Delta_2 : x \mapsto x \otimes 1 + 1 \otimes x + x \otimes x.\]
    \end{enumerate}
    
    If $\fg = \fn_1^2,$ and $A= k[x, y]/(x^p, y^p)$, then a complete set of representatives for the orbit space $\Gscr_A / \Aut(A)$ is given as follows:
    \begin{enumerate}\setcounter{enumi}{-1}
        \item $\widetilde G = \GG_{a(1)}^2$, with \begin{flalign*}
            \widetilde \Delta : x &\mapsto x \otimes 1 + 1 \otimes x\\
            y&\mapsto y \otimes 1 + 1 \otimes y,
        \end{flalign*}
        \item $G_1 = \GG_{a(2)}$, with \begin{flalign*}
             \Delta_1 : x &\mapsto x \otimes 1 + 1 \otimes x\\
            y&\mapsto y \otimes 1 + 1 \otimes y + \omega(x)
        \end{flalign*}
        \item $G_2 = \GG_{a(1)}\times \ZZ / p,$ with \begin{flalign*}
        \Delta_2 : x &\mapsto x \otimes 1 + 1 \otimes x\\
            y&\mapsto y \otimes 1 + 1 \otimes y + y \otimes y
        \end{flalign*}
        \item $G_3 = (\ZZ / p)^2$, with \begin{flalign*}
            \Delta_3 : x &\mapsto x \otimes 1 + 1 \otimes x + x \otimes x\\
            y&\mapsto y \otimes 1 + 1 \otimes y + y \otimes y.
        \end{flalign*}
    \end{enumerate}
\end{thm}

The following theorem is essentially a dualized version of the first theorem of Nguyen, L. Wang, and X. Wang \cite[Theorem~1.1]{NWW15}, which classifies all connected Hopf algebras of dimension $p^3$ with a one-dimensional primitive subspace. By fixing a local commutative algebra of dimension $p^3,$ we merely pick out which of the isomorphism classes listed in \cite{NWW15} have the relevant coalgebra structure, with some compatible commutative algebra structure. We find that there are only finitely many, up to isomorphism.

\begin{thm}\label{classification3}
    Let $\fg = \fn_3$ and $A = u(\fg) = k[x]/x^{p^3}.$ Then a complete set of representatives for the orbit space $\Gscr_A / \Aut(A)$ is given as follows:
    \begin{enumerate}\setcounter{enumi}{-1}
        \item $\widetilde G = \WW_{3(1)},$ with
        \[\widetilde \Delta : x \mapsto x \otimes 1 + 1 \otimes x,\]
        \item $G_1$, with
        \[ \Delta_1 : x \mapsto x \otimes 1 + 1 \otimes x + \omega(x^{p^2}),\]
        \item $G_2$, with
        \[\Delta_2 : x \mapsto x \otimes 1 + 1 \otimes x + \omega(x^p),\]
        \item $G_3 = \ZZ / (p^3)$, with
        \[\Delta_3 : x \mapsto x \otimes 1 +  1 \otimes x + x \otimes x. \]
    \end{enumerate}
\end{thm}
\begin{proof}
    The isomorphism classes of Hopf algebras in $\Gscr_A$ are each dual to some commutative Hopf algebra of dimension $p^3$ with a coalgebra isomorphic to $A^* = \cO(\WW_{3(1)})$, a connected coalgebra. We'll call this algebra structure canonical, and explicitly, we have
    \[A^* = k[dx, dy, dz] / (dx^p, dy^p, dz^p),\]
    where $dx, dy, dz$ are the dual vectors to $x, y = x^p, z = x^{p^2}$ respectively. We have fixed the comultiplication induced from generators via the canonical algebra structure described as
    \begin{flalign*}
        dx &\mapsto dx \otimes 1+ 1 \otimes dx\\
        dy &\mapsto dy \otimes 1 + 1 \otimes dy + \omega(dx)\\
        dz &\mapsto dz \otimes 1 + 1 \otimes dz + \omega(dy) + \omega(dx)[dy \otimes 1 + 1 \otimes dy]^{p-1}, 
    \end{flalign*}
    with all multiplications and powers also being defined canonically.
    Now, we have the primitive subspace $P(A^*) =  \langle dx \rangle$ is fixed since the coalgebra and unit for $A^*$ (the coaugmented coalgebra structure) is fixed. Thus any group scheme in $\Gscr_A$ has coordinate algebra isomorphic to one classified in \cite[Theorem~1.1]{NWW15}. By that theorem, there are precisely four distinct isomorphism classes of connected commutative Hopf algebras with a one-dimensional primitive subspace. Therefore, it suffices to verify that the four cocommutative expressions given in the theorem statement define Hopf algebra structures in $\Gscr_A$ and are not isomorphic to one another. We leave this as an exercise for the reader.  

\end{proof}
Also note, by the same classification \cite[Theorem~1.1]{NWW15}, we have that $k[x]/x^{p^3}$ admits an infinite family of nonisomorphic Hopf algebra structures which are not cocommutative when $p > 2$. For large $n$ it is clear that there exists local algebras over $k$, of dimension $p^n$, which admit infinitely many nonisomorphic compatible Hopf algebra structures. But it remains unclear whether there are always finitely many compatible cocommutative Hopf algebra structures up to isomorphism, on a given local algebra. 
\subsection{Abelian restricted enveloping algebras}
Now we extend the classifications from Section \ref{UniAbSection} to a result that applies to any abelian restricted Lie algebra of tame representation type, as well as any abelian restricted Lie algebra $\fg$, of finite representation type with principal block $B \subset u(\fg)$, having dimension $|B| \le p^3$. These are all of the cases for which Property PA is affirmed in Theorem \ref{mainThm}.
\begin{lemma}\label{artinianGroupScheme}
    Let $A$ be a local artinian commutative algebra over a perfect field $F$ and let $R$ denote the $n$-fold direct product of $A$ with itself. If $\Spec R$ has the structure of an abelian group scheme $G$ over $F$, then the connected component $G_0$ of the identity is isomorphic to $\Spec A$ as a scheme, and we have $G \cong G_0 \times H$ for some finite abelian group $H$ order $n$.
\end{lemma}
\begin{proof}
    Since $F$ is perfect, we have automatically that $G_{\mathrm{red}}$ and $G_0$ are both closed subgroups of $G,$ with $G_0 \cong \Spec A.$ We have $G_{\mathrm{red}} = \Spec R_{\mathrm{red}}$ is the $n$-fold disjoint union of $\Spec F$ with itself, since $A_{\mathrm{red}} \cong F$. Then the underlying topological space for the scheme $G_{\mathrm{red}}$ is actually a discrete topological group induced by the scheme theoretic structure maps, and we have $H = G_{\mathrm{red}}$ is a constant group scheme of order $n$. 
    Since $H$ is also the group of connected components, we have now a split short exact sequence of abelian group schemes
    \[G_0 \hookrightarrow G \twoheadrightarrow H\] and therefore $G \cong G_0 \times H.$
\end{proof}

Now suppose $\fg$ is an abelian restricted Lie algebra. By Theorem \ref{SeligmanStructure}, we have $\fg \cong \ft^r \oplus \Nscr(\fg)$ for some $r \ge 0,$ and therefore there is an isomorphism of associative algebras $u(\fg) \cong \prod_{i=1}^{p^r}u(\Nscr(\fg))$. Since $u(\Nscr(\fg))$ is a finite local commutative algebra over an algebraically closed field $k$, we can apply Lemma \ref{artinianGroupScheme}. For $A = u(\fg)$ and $B = u(\Nscr(\fg))$ the principal block, we therefore have that a complete set of representatives for the orbit space $\Gscr_A / \Aut(A)$ can be classified by their Cartier duals as $\Pscr_r\times (\Gscr_B / \Aut(B)),$ where $\Pscr_r$ is the set of partitions of $r$, i.e. the set of isomorphism classes of abelian groups of order $p^r$. As a corollary, whenever $\Nscr(\fg)$ is a unipotent algebra of the type covered by Theorems \ref{TOClassification1}, \ref{WangClassification2}, \ref{classification3}, we can give a complete set of representatives of $\Gscr_A / \Aut(A)$ explicitly. Instead we'll present what we need out of this classification in the following proposition. 
\begin{prop}\label{PAnullcone}
    Let $\fg$ be an abelian restricted Lie algebra with infinitesemimal group scheme $\widetilde G.$ Let $\fh = \Nscr(\fg)$, with infinitesimal group scheme $\widetilde H$. Then $\widetilde H$ has Property PA if and only if $\widetilde G$ has Property PA.
\end{prop}
\begin{proof}
    We have $\fg \cong \ft^r \oplus \fh$ for some $r \ge 0,$ and hence $u(\fg) \cong \prod_{i = 1}^{p^r}u(\fh)$. Let $A = u(\fg)$ and $B = u(\fh)$, and so $A = u(\ft^r)\otimes B.$ 
    Let $G \in \Gscr_A$ be a group scheme structure. The Cartier dual $G^*$ is therefore isomorphic to $Z \times H^*,$ where $H^*$ is dual to some $H \in \Gscr_B$, and $Z$ is a finite abelian group of order $p^r.$ Further, we have that $Z = \chi(G)$, identifying simple representations of $G$, orthogonal idempotents of $A$, and connected components of $G^*.$ 

    For any such $G, H$ we have, by the K{\"u}nneth theorem, an isomorphism of graded commutative algebras $H^*(G, k) \cong H^*( H, k)$, so we naturally identify the projective varieties \[ \Proj H^*(A,k) = \Proj H^*(B,k),\] and $\Xscr$ the set of closed points. Thus the supports of $A$-modules and $B$-modules are both regarded as subspaces of $\Xscr$.

    Now, whenever $e$ is a simple $u(\ft^r)$-module and $M$ any $B$-module, we write $eM$ to mean the module $e \otimes M$ over $u(\ft^r) \otimes B = A$. In this case we will also abuse notation and write $e = ek$ for the trivial $B$-module $k$, so that $e$ is also a simple $A$-module. Since $B$ is local, each simple $A$-module is $e$ for some simple $u(\ft^r)$-module. If $e_1, e_2$ are simple $A$-modules, we will write $e_1e_2$ to mean the product in the character group $\chi(G)$ whenever a choice of $G \in \Gscr_A$ is made clear.
    
    Under these identifications, for each point $\fp \in \Xscr,$ an indecomposable $A$ module $M$ with support $\{\fp\}$ is $ eM'$, for an indecomposable $B$-module $M'$ with support $\{\fp\}$, and some simple $A$-module $e.$ If $\otimes$ is the product for a group scheme $G = \chi(G)^* \times H\in \Gscr_A$, we have 
    \[M_1 \otimes M_2 = (e_1M_1') \otimes (e_2M_2') = (e_1e_2)M_1'\otimes M_2'\]
    for indecomposable $M_i = e_iM_i',$ where $e_1e_2$ is the product in $\chi(G)$ and $M_1'\otimes M_2'$ is the corresponding product for $H.$

    Let $\fp \in \Xscr$ be a noble point for $H$. Then there is a subgroup $\iota : C \hookrightarrow H$ with $\fp = \sqrt{\ker H^*(\iota)}$. Since $G = \chi(\widetilde G)^* \times H$, we can identify $H = {\bf 1} \times H < G$ and $\iota' : C \hookrightarrow G$ the composition. Then see that $\fp = \sqrt{\ker H^*(\iota')}$ so $\fp$ is noble for $G$ too. Since $\pi$-points factor through unipotent subgroups, the converse holds as well: any $C \hookrightarrow G$ will factor through $H$, the unique maximal unipotent subgroup, hence noble points for $G$ are noble for $H$. 

    Now suppose $\widetilde H$ has Property PA. To show that $\widetilde G$ has Property PA, it's sufficient to start with some $H \in \Gscr_B$, and assume that $G = \chi(\widetilde G)^* \times H,$ so that $G^* = \chi(\widetilde G) \times H^*,$ and show that $\widetilde G$ and $G$ are in noble correspondence. Since we assumed $\widetilde H$ is unipotent and has Property PA, we know that $\widetilde H$ and $H$ are in noble correspondence.

    Let $M_i = e_i M_i'$ be an indecomposable $A$-module, with $e_i$ simple and $M_i$ indecomposable over $B$, for $i = 1, 2$, with $\supp_A(M_i) = \{\fp\}$. Then also $\supp_B(M_i') = \{\fp\}$. If $\fp$ is noble for $G$, we see also that $\fp$ is noble for $H$. Since $\chi(G) = \chi(\widetilde G)$ by assumption, we have
    \begin{flalign*}
        M_1 \otimes M_2 = (e_1M_1')\otimes(e_2M_2') &= (e_1e_2)M_1'\otimes M_2'\\
        &\cong (e_1e_2)M_1' \widetilde \otimes M_2'\\
        &= (e_1M_1')\widetilde\otimes (e_2M_2') = M_1 \widetilde \otimes M_2, 
    \end{flalign*}
    where the product of simple module $e_1e_2$ is unambiguous.
    The converse follows in a similar fashion. 
\end{proof}
\section{Noble correspondence}
In this section we prove the following three part theorem. Parts 1, 2, 3, are proven in the subsections \ref{finiteSection}, \ref{wildSection}, \ref{tameSection}, respectively. 
\begin{thm}\label{mainThmUnip}
    Let $\fg$ be a unipotent abelian restricted Lie algebra and $\widetilde G$ its infinitesimal group scheme.
    \begin{enumerate}
        \item If $\widetilde G$ is of finite representation type and $|\widetilde G| \le p^3$ then $\widetilde G$ has Property PA. 
        \item If $\widetilde G$ is of wild representation type then $\widetilde G$ does not have Property PA. 
        \item If $\widetilde G$ is of tame representation type, i.e. $p = 2$ and $\fg = \fn_1^2,$ then $\widetilde G$ has Property PA. 
    \end{enumerate}
\end{thm}
By Theorem \ref{RepTypeClassification} and Proposition \ref{PAnullcone}, we have that our main Theorem \ref{mainThm} follows immediately from the theorem above.

\subsection{Finite representation type algebras}\label{finiteSection}
Let $\fg$ be a unipotent abelian restricted algebra of finite representation type, with infinitesimal group scheme $\widetilde G,$ and $A = u(\fg).$ By Theorem \ref{RepTypeClassification}, we have $\fg = \fn_n$ for some $n,$ so we will take $A = k[x]/x^{p^n}$. A complete set of orbit representatives for $\Gscr_A / \Aut(A)$ is given in the case $n = 1$ in Theorem \ref{TOClassification1}, in the case $n = 2$ in Theorem \ref{WangClassification2}, and in the case $n = 3$ in Theorem \ref{classification3}. In any case, the spectrum $\Proj H^*(A)$ consists of a single closed point and it's always noble. To show that $\widetilde G$ and $G$ are in noble correspondence for each $G \in \Gscr_A,$ we must therefore show the Green ring is completely invariant. It is sufficient to show that there is an equality of Green rings $\fR(\widetilde G) = \fR(G)$ for each $G$ appearing in the classifications \ref{TOClassification1}, \ref{WangClassification2}, \ref{classification3}, as any $G' \in \Gscr_A$ is in the form $G^\varphi$ for some $G$ as classified, and some $\varphi \in \Aut(A).$ As such, since there is a unique indecomposable $A$-module of each dimension $i$ for $1 \le i \le p^r,$ we have $M^\varphi \cong M$ for any $M$ and for any $\varphi.$ Then part 1 of Theorem \ref{mainThmUnip} follows from the identity \ref{phiDist}. 

We denote $J_i$ the unique indecomposable $A$-module of dimension $i$. For any $G  \in \Gscr_A,$ with tensor product $\otimes$ of $A$-modules, we denote $c_{i, j, \ell}(G)$ the \textit{relative Clebsch-Gordon coefficients,} such that
\[J_i \otimes J_j  \cong \sum_{\ell = 1}^{p^r} c_{i, j, \ell}(G)J_{\ell}.\]
One finds indeed that $c_{i, j, \ell}$ is invariant of $G$ in every relevant case. 
\begin{prop}\label{cyclicPTensors}
    Suppose $n = 1$ as in Tate and Oort's classification Theorem \ref{TOClassification1} \cite{OT1970}. Then $\fR(\widetilde G) = \fR(\widetilde G_1)$ and hence $\widetilde G$ and $\widetilde G_1$ are in noble correspondence. 
\end{prop}
\begin{proof}
    It is classical that for $i  \le j$ and $p < i + j$, we have
    \[c_{i, j, \ell}(\widetilde G) = c_{i, j, \ell}(G) = 
    \begin{cases} 
    i + j - p, & \ell = p\\
    1, & \ell = j - i + 1 +  2m,\quad  0 \le m \le p-j -1,\\
    0 & \text{otherwise},
    \end{cases}\]
    and that for $i \le j$ and $p \ge i + j$, we have
    \[c_{i, j, \ell}(\widetilde G) = c_{i, j, \ell}(G) = 
    \begin{cases} 
    1, & \ell = j - i + 1 +  2m,\quad  0 \le m \le i - 1,\\
    0 & \text{otherwise}.
    \end{cases}\]
    These formulas in the case of the cyclic group $G$ are implicit in the work of Green \cite{Green62} in 1962. For the infinitesimal $\widetilde G$, the same formulas were calculated implicitly from restricted representations of $\sl_2$ in characteristic $ p > 2$ as early as 1981 by Benkart and Osborn in \cite{BenkartOsborn82}.
\end{proof}
Now the following propositions determine Green rings inductively, by restricting the group schemes classified in \ref{WangClassification2}, \ref{classification3} down to known subgroups with algebras generated by the powers $x^{p}$, $x^{p^2}, $ and leveraging Frobenius reciprocity. We omit the proofs, as even stating a formula for Clebsch Gordon coefficients $c_{i,j,\ell}(G)$ takes considerable space. The calculation, like those in Proposition \ref{cyclicPTensors}, are a generalization of Green's original calculation for what we would call $\fR(\ZZ / (p^n)) \otimes_\ZZ \ZZ/p$ in \cite{Green62}. The only part that is not classical is verifying which subgroups exist for any cocommutative Hopf algebra structure, which uses the classifications \ref{WangClassification2}, \ref{classification3} explicitly. 
\begin{prop}
    Suppose $n = 2$ as in X. Wang's classification Theorem \ref{WangClassification2} \cite{XWang13}. Then $\fR(\widetilde G) = \fR(G_1) = \fR(G_2)$ and hence $\widetilde G$ is in noble correspondence with both $G_1$ and $G_2.$ 
\end{prop}

\begin{prop}
    Suppose $n = 3$ as in Theorem \ref{classification3}. Then $\fR(\widetilde G) = \fR(G_i)$ for $i = 1,2,3.$ Hence $\widetilde G$ is in noble correspondence with $G_i$ for $i = 1, 2, 3. $
\end{prop}
\subsection{Wild representation type algebras}\label{wildSection}
Let $A$ be a groupable augmented algebra. We can define an action of $\Aut(A)$ on the closed points $\Xscr(A) \subset \Proj H^*(A, k)$ as follows: supposing $A = kG,$ we have any given $\fp \in \Xscr(A)$ is $\fp(\alpha)$ for some $\pi$-point $\alpha : k[t]/t^p \to A$ for $G$. So for an augmented automorphism $\varphi : A \to A$ we define $\varphi \cdot \fp = \fp(\varphi \circ \alpha),$ using that the composition $\varphi \circ \alpha$ is a $\pi$-point for $G^{\varphi}$, hence a generalized $\pi$-point for $A$. We have indeed whenever $\alpha \sim \beta$ as $\pi$-points that $\varphi\circ \alpha \sim\varphi \circ \beta$ as generalized $\pi$-points and hence $\varphi \cdot \fp$ is well defined. 
It is straightforward to verify whenever $\fp$ is noble for $G \in \Gscr_A$, that $\varphi\cdot \fp$ is noble for $G^{\varphi}$, and further, whenever $\fp \in \Xscr(A, M)$ for a module $M$, that $\varphi\cdot \fp  \in \Xscr(A, M^{\varphi})$ by picking representative $\pi$-points. 

For $\fp \in \Xscr$ we denote the isotropy subgroups by
\begin{flalign}\label{isotropy} \Omega(A, \fp) &= \Aut(A)^{\fp} = \{\varphi \mid \varphi \cdot \fp = \fp\},\\
\Omega(A, \chi ) &= \{ \varphi \mid e^{\varphi} \cong e \text{ for each simple $A$-module $e$} \}.
\end{flalign}
Whenever $\varphi \in \Omega(A, \chi),$ we have that $\chi(G) = \chi(G^{\varphi})$ for any $G \in \Gscr_A.$
The following lemma is adapted from the unpublished \cite{Bloom24}. 
\begin{lemma}\label{twistNPALemma}
    Let $\fg$ be a restricted Lie algebra, let $\widetilde G$ be the infinitesimal group scheme, and $A = u(\fg).$ Suppose $x \in \Nscr(\fg)$ is of order $r\ge 1$, with $[x, \fg] = 0.$ Suppose for $\fp = \fp(x^{[p]^{r-1}})$ that there exists an isotropy $\varphi \in \Omega(A, \fp )$ such that $(J_1\uparrow_{\langle x \rangle}^\fg)^{\varphi^{-1}}\downarrow_{\langle x \rangle}^{\fg}$ is not isomorphic to $nJ_{p^s}$ for any $n, s \ge 0.$ Then $\widetilde G$, $\widetilde G^{\varphi}$ are not in noble correspondence. If $\varphi \in \Omega(A, \chi)$ then $\widetilde G$ does not have Property PA.
\end{lemma}
\begin{proof}
    We'll denote $V = J_1\uparrow_{\langle x \rangle\fg}$ the induced module. The PBW theorem let's us conclude that $\Xscr(A,V) \subset \{\fp\}$ for any $x \in \Nscr(\fg)$.  Since $[x, \fg] = 0$ and $u(\fg)$ is a free module over the subalgebra $u(\langle x \rangle )$ (see Nichols and Zoeller \cite{NZ89}), we have that $V\downarrow^{\fg}_{\langle x \rangle} = [\fg : \langle x \rangle ]J_1,$ where $[\fg : \langle x \rangle]$ is the index $p^{n - r},$ with $n = \dim \fg.$ We conclude then that $V\downarrow^{\fg}_{\langle x^{[p]^{r-1}} \rangle} = V\downarrow^{\fg}_{\langle x \rangle }\downarrow^{\langle x \rangle}_{\langle x^{[p]^{r-1}} \rangle}$ is not projective and hence $\supp_A(V) = \{\fp\}.$
    We have by Frobenius reciprocity that
    \[V \widetilde \otimes V \cong V\downarrow_{\langle x \rangle}^{\fg}\uparrow_{\langle x \rangle}^{\fg} \cong [\fg :\langle x \rangle ]V.\]
    But one finds that since $V^{\varphi^{-1}} \downarrow^{\fg}_{\langle x \rangle}$ is not in the form $nJ_{p^s},$ there must be non-isomorphism
    \[(V^{\varphi^{-1}} \widetilde \otimes V^{\varphi^{-1}})\downarrow^{\fg}_{\langle x \rangle} \cong V^{\varphi^{-1}}\downarrow^{\fg}_{\langle x \rangle}  \widetilde \otimes V^{\varphi^{-1}}\downarrow^{\fg}_{\langle x \rangle}  \not\cong [\fg : \langle x \rangle ] V^{\varphi^{-1}}\downarrow^{\fg}_{\langle x \rangle} .\]
    It follows that $V^{\varphi^{-1}} \widetilde \otimes V^{\varphi^{-1}} \not\cong [\fg : \langle x \rangle ] V^{\varphi^{-1}}$ and hence
    \[V \widetilde \otimes^{\varphi} V \not\cong [g : \langle x \rangle ] V \cong V\widetilde\otimes V\] by the identity \ref{phiDist}. 
    Supposing $\varphi \in \Omega(A, \fp)$ we know $\fp $ is noble for $\widetilde G$ and hence $\varphi\cdot\fp  = \fp $ is noble for $\widetilde G^{\varphi}.$ Hence $\widetilde G$ and $\widetilde G^{\varphi}$ are not in noble correspondence. 
    Supposing $\varphi \in \Omega( A, \chi)$ we have $\chi(\widetilde G) = \chi(\widetilde G^{\varphi})$, and conclude $\widetilde G$ does not have Property PA. 
\end{proof}

\begin{prop}\label{primNPA}
    Let $\fg$ be an abelian restricted Lie algebra of wild representation type with no nontrivial wild direct summands (c.f. Corollary \ref{primordial}), and let $A = u(\fg).$
    \begin{enumerate}[I.]
    \item For $p = 2$:
        \begin{itemize}[label = {--}]
        \item If $\fg = \fn_1\oplus \fn_1 \oplus \fn_1$, say $A = k[x,y,z]/(x^p, y^p, z^p)n$ and let 
        \[\varphi(x) = x + yz,\qquad \varphi(y) = y,\qquad \varphi(z) = z.\]
        \item If $\fg = \fn_n \oplus \fn_m$ for $ m > n \ge 1,$ say $A = k[x,y]/(x^{p^n}, y^{p^m}),$ and let
        \[ \varphi(x) = x + y^2,\qquad \varphi(y) = y\]
        \item If $\fg = \fn_n \oplus \fn_n$ for $n \ge 2,$ say $A = k[x,y]/(x^{p^n}, y^{p^n}),$ and let
        \[\varphi(x) = x + y^{p^{n-1}-1}, \qquad \varphi(y) = y.\]
        \end{itemize}
    \item For $p > 2,$ if $\fg = \fn_n \oplus \fn_m$ for $n, m \ge 1$, say $A = k[x,y]/(x^{p^n}, y^{p^n},$ and let 
    \[\varphi(x) = x+y^2,\qquad \varphi(y) = y.\]
    \end{enumerate}
    In each case, we have $x \in \Nscr(\fg)$ (of order $r$) and $\varphi$ defines an isotropy in $\Omega(A, \fp)$ for $\fp = \fp(x^{[p]^{r - 1}})$. Further $(J_1\uparrow_{\langle x \rangle}^\fg)^{\varphi^{-1}}\downarrow_{\langle x \rangle}^{\fg}$ is not isomorphic to $nJ_{p^s}$ for any $n, s \ge 0.$ We conclude for $\widetilde G$ the infinitesimal group scheme for any such $\fg$, that $\widetilde G$ does not have Property PA. 
\end{prop}

\begin{cor}
    Let $\widetilde G$ be the infinitesimal group scheme for a unipotent abelian restricted Lie algebra. Then $\widetilde G$ does not have Property PA. 
\end{cor}
\begin{proof}
    Suppose $\fg = \fg' \oplus \fg''$ for an abelian restricted Lie algebra $\fg'$ of wild representation type with no nontrivial wild direct summands. Then $A = u(\fg) = u(\fg')\otimes u(\fg''). $ It is straightforward now to show that the isotropy $\varphi \in \Omega(u(\fg'), \fp)$ in Proposition \ref{primNPA} extends to an isotropy in $\Omega( A, \fp)$, which lets us apply Lemma \ref{twistNPALemma}.
\end{proof}

\subsection{The Ba{\v s}ev-Conlon formula}\label{tameSection}
Let $k$ be an algebraically closed field of characteristic $p = 2$. We let $\fg = \fn_1\oplus \fn_1$ and $A  = u(\fg) = k[x,y]/(x^2, y^2).$ We will denote by $P$ the regular representation, i.e. $A$ as a left module over itself.  
We begin by defining the indecomposable modules $V_{2n}(\fp)$ for closed points $\fp \in \PP^1 = \Xscr(A)$ in $ \Proj H^*(A, k).$ By $\fp = [a : b] \in \PP^1$, we will always mean $\fp = \fp(\alpha)$ for the $\pi$-point 
\[\alpha : k[t]/t^2 \to A\] defined by taking $t \mapsto ax + by,$ which we note is well defined up to linear equivalence. 

The indecomposable modules over the Kronecker algebra $A$ were classified by Ba{\v s}ev \cite{Basev61} in 1961 as representations over $k$ of the Klein 4-group. They are attributed back to Kronecker's original work on quadratic forms appearing between 1890-96 \cite{Kronecker1890}. We present the following theorem which lists all finite $A$-modules having support containing exactly one closed point in $\Xscr(A) = \PP^1.$

\begin{thm}\label{kroneckerClassificationThm}
    Let $\fp \in \Xscr(A) = \PP^1$ be the point $[a : b]$ in our coordinate system, and define $s_2 = ax + by \in A$ and $s_1 = cx + dy$ such that $s_1, s_2$ forms a basis for the subspace $\langle x , y\rangle \subset A.$
    
    Let $M$ denote a vector space of dimension $2n$, with $k$-linear decomposition into \emph{lower and upper} blocks $M = M_\ell \oplus M_{u}$, with $M_\ell, M_u$ each of dimension $n$. We let $V_{2n}(\fp) = M$ denote the $A$-module defined by the following matrix representations of the actions of $s_1, s_2 \in A$
    \[s_1 : \begin{pmatrix}0 & I_n\\0 &0\end{pmatrix},\quad s_2 :
    \begin{pmatrix}0 &\mathfrak{N}_n\\0 &0    \end{pmatrix}, \]
    where $I_n$ is the diagonal ones matrix and $\mathfrak{N}_n$ is an upper triangular nilpotent Jordan block of rank $n - 1.$

    Then we have that
    \begin{enumerate}[I.]
        \item The module $V_{2n}(\fp)$ is, up to isomorphism, not dependent on choice of $a, b, c, d,$ such that $\fp = [a : b] \in \PP^1$, and such that $s_1 = cx + dy$ and $s_2$ are linearly independent,
        \item The module $V_{2n}(\fp)$ is indecomposable,
        \item The support $\Xscr(A, V_{2n}(\fp))$ is $\{\fp\} \subset \PP^1,$ and
        \item Any finite indecomposable module $V$ with support $\Xscr(A,V) = \{\fp\}$ is of even dimension $2n$ and is isomorphic to $V_{2n}(\fp)$, for some $n$. 
    \end{enumerate}
\end{thm}

The main result for this section is the \textit{Ba{\v s}ev-Conlon formula for the noble square}, which holds for an arbitrary choice of $\otimes$ coming from a Group algebra structure in $G \in \Gscr_A.$
\begin{equation}\label{nobleSquare}
    V_{2n}(\fp) \otimes V_{2n}(\fp) \cong 2V_{2n}(\fp) + (n^2 - n)P \quad \forall \text{ noble } \fp \in \PP^1\text{ for }G. 
\end{equation}
As a historical note, we have named formula \ref{nobleSquare} after Ba{\v s}ev and Conlon jointly because Conlon's paper \cite{Conlon65} offered a correction to Ba{\v s}ev's formulas presented in \cite{Basev61} in the case of the Klein 4 group as a point in $\Gscr_A.$ We find actually that the set of noble points in this case coincides precisely with the modules where Conlon's formulas agree with Ba{\v s}ev's. By following Ba{\v s}ev's original ring-theoretic argument (Lemmas \ref{ProjComponent}, \ref{BasevLinearFunctionalLemma} below), we will see that the isomorphism types of squares $V_{2n}(t) \otimes V_{2n}(t)$ are sufficient for determining the entire Green ring structure. We have ommitted any discussion of nonprojective indecomposables having support other than a closed singleton $\{\fp\}$, because these are not relevant to our Property PA. Still, we would be remiss not to mention that all such modules are Heller shifts of the trivial module and tensor with all other modules in a way that does not depend on the coalgebra structure for $A$!  

We will fix notation from Theorem \ref{WangClassification2} throughout, that a complete set of orbit representatives for $\Gscr_A /\Aut(A)$ is given by $\widetilde G, G_1, G_2, G_3$, with comultiplications $\widetilde \Delta, \Delta_1, \Delta_2, \Delta_3$ and products $\widetilde \otimes, \otimes_1, \otimes_2, \otimes_3$ respectively. To show that $\widetilde G$ is in noble correspondence with a given $G \in \Gscr_A,$ we first argue how $\widetilde G$ is in noble correspondence with $G_1, G_2, G_3$ (it is tautological that $\widetilde G$ is in noble correspondence with itself). Then for a given $G \in \Gscr_A,$ we 
identify the orbit representative $G' = \widetilde G, G_1, G_2, G_3$, with product $\otimes'$ and $\varphi \in \Aut(A)$ with $G^{\varphi} = G'.$
Then for a fixed point $\fp$ which is noble for $G$, we have $\varphi \cdot \fp$ is noble for $G'.$ Given any indecomposable $M, N$ with support $\{\fp\}$, say $M \cong V_{2m}(\fp), N \cong V_{2n}(\fp),$ we have $M^\varphi, N^\varphi$ are indecomposable with support $\{\varphi\cdot \fp\},$ and hence by Theorem \ref{kroneckerClassificationThm} we have $M^{\varphi} \cong V_{2m}(\varphi\cdot \fp), N^{\varphi} \cong V_{2n}(\varphi\cdot\fp)$. Since $\widetilde G$ is in noble correspondence with $G'$ and $\varphi \cdot \fp$ is noble for $G'$, we have
\[M^\varphi \otimes' N^\varphi \cong M^{\varphi} \widetilde \otimes N^\varphi,\]
and hence
\[M \otimes N \cong (M^\varphi \widetilde \otimes N^\varphi)^{\varphi^{-1}}\] by identity \ref{phiDist}. The tensor product property for $\supp_A$, with Theorem \ref{kroneckerClassificationThm}, shows that there is a decomposition
\[V_{2m}(\fp) \widetilde\otimes V_{2n}(\fp) = c_{m, n, P}(\fp) P \oplus \sum_{\ell} c_{m, n,\ell}(\fp) V_{2\ell}(\fp),\]
where $P$ is unique indecomposable projective module. We will show further that the coefficients $c_{m, n, P}(\fp), c_{m, n,\ell}(\fp)$ depend only on $m,n$ and not which noble $\fp$ is chosen. It follows that $\widetilde G$ and $\widetilde G^{\varphi}$ are in noble correspondence for any $\varphi \in \Aut(A)$, since
\[V_{2m}(\fp)\widetilde\otimes^{\varphi} V_{2n}(\fp)  \cong (V_{2m}(\varphi^{-1} \cdot \fp)\widetilde\otimes V_{2n}(\varphi^{-1} \cdot \fp))^{\varphi}. \]
Hence $M\otimes N \cong M \widetilde\otimes^{\varphi^{-1}}N \cong M \widetilde \otimes N$ for indecomposable $M, N$ of support $\{\fp\}$, and we conclude $\widetilde G$ is in noble correspondence with $G.$ 

The following lemmas are proven in detail in the author's unpublished \cite{Bloom24}, but essentially generalize, to any $G \in \Gscr_A$, the same ring-theoretic argument outlined originally by Ba{\v s}ev \cite{Basev61}. 

\begin{lemma}\label{ProjComponent}
    Let $n \le m$, and $\fp \in \PP^1$ be any closed point. Let $\otimes$ be the product for some group $G \in \Gscr_A$. 
    Then 
    \[V_{2n}(\fp) \otimes V_{2m} \cong V \oplus (mn - n)P,\]
    where $V = \sum c_{n,m, \ell }(\fp)V_{2\ell}(\fp)$,
    for Clebsch-Gordon coefficients $c_{n, m, \ell}(\fp)$ relative to $G$, satisfying $\sum c_{n, m, \ell}(\fp)2\ell = 4n.$ Equivalently, we may say $c_{n, m , P}(\fp) = mn - n.$
\end{lemma}

\begin{lemma}\label{BasevLinearFunctionalLemma}
Let $\otimes$ be the product for some group in $G \in \Gscr_A,$ and for any $n \le m,$ and closed point $\fp \in \PP^1,$ let $c_{n, m, \ell}(\fp)$ be the Clebsch-Gordon coefficients relative to $G$ as above. Then
\begin{enumerate}[I.]
    \item For any $\fp,$ we have $c_{n, m, \ell}(\fp) = 2\delta_{n, \ell}$ whenever there is inequality $n < m,$ and
    \item For any $\fp$, there exists a subset $N(\fp) \subset \NN$, such that no two consecutive numbers are elements of $N(\fp)$, and such that
    \[c_{n, n , \ell}(\fp) = \begin{cases}
        2\delta_{n, \ell} & n \not \in N(\fp),\\
        \delta_{n-1,\ell} + \delta_{n+ 1, \ell} & n \in N(\fp).
    \end{cases}\]
\end{enumerate}
\end{lemma}

Now we see the Ba{\v s}ev-Conlon formula \ref{nobleSquare} is stating that, for any noble point $\fp \in \PP^1$ for a fixed $G \in \Gscr_A$, we have $N(\fp) = \emptyset.$ To show this, we first examine which points are noble in the fixed homogeneous coordinate system for $\PP^1$ stated in \ref{kroneckerClassificationThm}.
\begin{example}
    Let $G$ be an element of the complete set of representatives $\{\widetilde G, G_1, G_2, G_3\}$ for $\Gscr_A / \Aut(A).$  
    \begin{enumerate}\setcounter{enumi}{-1}
        \item For $G = \widetilde G$, each $\fp \in \PP^1$ is noble, as $\langle ax + by\rangle$ always defines a Lie subalgebra of $\langle x, y\rangle.$
        \item For $G = G_1,$ the only noble point is $[1 : 0].$ This is because $\GG_{a(2)}$ has only one nontrivial proper closed subgroup, and this subgroup is isomorphic to $\GG_{a(1)}$, and the inclusion of group algebras comes from $t \mapsto x.$
        \item For $G = G_2,$ we have $G \cong \GG_{a(1)} \times \ZZ / 2.$ One sees in any characteristic that the only nontrivial proper subgroups of $\GG_{a(1)} \times \ZZ / p$ are $\GG_{a(1)} \times 0$ and $0 \times \ZZ / p$, with group algebra inclusions given by $t \mapsto x$ and $t\mapsto y$ respectively. Hence the only noble points for $G_2$ are $[1 : 0]$ and $[0 : 1].$
        \item For $G = G_3,$ there are three nontrivial proper subgroups of $(\ZZ / 2)^2$, with group algebra inclusions given by $t \mapsto x, t\mapsto x + y + xy,$ and $t \mapsto y.$ By \cite[Proposition 2.2]{FrPev05}, only the linear terms matter, and we have the noble points for $G_3$ are therefore $[1: 0], [1 : 1],$ and $[0 : 1].$
    \end{enumerate}
\end{example}

\begin{definition}
    Let $\fp$ be fixed and denote $V_{2n} = V_{2n}(\fp)$. For $n, m \in \NN$ we define the \emph{canonical extension}
    \[0 \to V_{2n} \to V_{2(n + m)} \to V_{2m} \to 0\]
    of $V_{2m}$ by $V_{2n}$, with \emph{canonical mono} $\mu_{n,m} : V_{2n} \hookrightarrow V_{2(n + m)}$ defined by the induced map from including ordered bases for respective upper and lower blocks (\ref{kroneckerClassificationThm}),  and \emph{canonical epi} $\epsilon_{n, m} : V_{2(n + m ) } \twoheadrightarrow V_{2m}$ by the quotient defined on the ordered bases for upper and lower blocks. It is clear the canonical monos and epis are $A$-linear and define an extension of $A$-modules.  
\end{definition}
The following proposition follows from direct examination of the modules in Ba{\v s}ev's Theorem \ref{kroneckerClassificationThm}. 
\begin{prop}\label{classifyingMonos}
    Let $\nu : V_{2n}(\fp) \hookrightarrow V$ be any monomorphism of $A$-modules, with $V = \sum_\ell c_\ell V_{2\ell}(\fp)$ a finite module. Then there exists an $\ell' \ge n,$ with direct summand inclusion $V_{2\ell'}(\fp)\hookrightarrow V,$ and an automorphism $f : V \to V$ of $A$-modules, such that the diagram 
\[\begin{tikzcd}
	{V_{2n}(\fp)} & {V_{2\ell'}(\fp)} & V \\
	&& V
	\arrow["{\mu_{n, \ell' - n}}", hook, from=1-1, to=1-2]
	\arrow["\nu"', hook, from=1-1, to=2-3]
	\arrow[hook, from=1-2, to=1-3]
	\arrow["f", dashed, from=1-3, to=2-3]
\end{tikzcd}\]
    commutes. In other words there is precisely one isomorphism class of monos $V_{2n}\hookrightarrow V$ for each $\ell \ge n$ with $c_\ell \neq 0$. 
\end{prop}

\begin{proof}
    Fix $\fp$ and denote $V_{2n} = V_{2n}(\fp).$ Let $V$ be a finite dimensional $A$-module with $\supp(V) = \{\fp\}$, and with no projective direct summand. Let $\nu : V_{2n} \hookrightarrow V$ be any mono. We have $V' = \coker(\nu)$ has support $\{\fp\}$ \cite{FrPev05}, and has no projective direct summand.
    Computation with a minimal resolution of each $V_{2n}$ shows that $\Ext^1_A(V_{2n}, V_{2m})$ is a natural subquotient of $nV_{2m}$. The resultant complex has basis of cocycles given by the lower block of each $V_{2m}$, plus the first basis element of the upper blocks. Reducing to cohomology, the first basis element of the upper blocks, and last basis element of the lower blocks give a basis for $\Ext^1_A(V_{2n}, V_{2m})$.

    Write $V' = \sum_\ell d_{\ell} V_{2\ell},$ so we have
    $\Ext^1_A(V', V_{2n}) = \sum_\ell d_{\ell}\Ext^1_A(V_{2\ell}, V_{2n})$. 
    Then an arbitrary linear Baer sum combination of extensions represented by our choice of cocycles in each copy of $V_{2n}$ shows that any with center term isomorphic to $V$ has monomorphism 
    $V_{2n} \hookrightarrow V$ in the form claimed. In particular, the assumed extension 
    \[V_{2n} \xrightarrow{\nu} V \to V' \in \Ext^1_A(V', V_{2n})\] has $\nu$ of the form claimed. 
\end{proof}


Now we give an original proof of the Ba{\v s}ev-Conlon formula, using nobility in an essential way. This was first provided in the unpublished \cite{Bloom24}.
\begin{thm}\label{mainSqTheorem}
    Let $G$ be an element of the complete set of representatives $\{ \widetilde G, G_1, G_2, G_3\}$ for $\Gscr_A/\Aut(A),$ and $c_{n, m, \ell}(\fp)$ be the Clebsch-Gordon coefficients relative to $G$. Let $\fp \in \PP^1$ be noble for $G$. Then $c_{n, n, \ell}(\fp) = 2\delta_{n, \ell}$ for each $n, \ell,$ i.e. \[V_{2n}(\fp )\otimes V_{2n}(\fp) \cong 2V_{2n}(\fp) \oplus (n^2 -n)P,\]
    where $\otimes$ is the product for $G$.
\end{thm}
\begin{proof}
    Suppose $\fp = \fp(\alpha)$ for the inclusion $\alpha : kC \to A$ of group algebras, for some $C < G.$ 
    
    Then for $n = 1,$ we have
    \[(V_{2}(\fp) \otimes V_{2}(\fp))\downarrow_{C}^G \cong V_{2}(\fp)\downarrow^G_C \otimes V_{2}(\fp)\downarrow^G_C,\]
    which by direct computation gives $(V_{2}(\fp) \otimes V_{2}(\fp))\downarrow_{C}^G \cong 4J_1$ for any case of $G$. Since also for any case of $G$, we see $V_4(\fp)\downarrow^G_C \cong J_2 \oplus 2J_1$,  we conclude that
    \[V_2(\fp) \otimes V_2(\fp) \cong 2V_2(\fp).\]

    For fixed noble $\fp$ for $G$, we denote $V_{2n} = V_{2n}(\fp).$
    Suppose for contradiction that for some $n > 1,$ we have $V_{2n} \otimes V_{2n} \cong V_{2(n-1)} \oplus V_{2(n + 1)} \oplus (n^2 - n )P.$ By Lemma \ref{BasevLinearFunctionalLemma}, we know then that $V_{2(n-1)} \otimes V_{2(n-1)} \cong 2V_{2(n-1)} \oplus (n^2 - 3n + 2)P$ and that $V_{2(n+1)} \otimes V_{2(n+1)} \cong 2V_{2(n+1)}\oplus (n^2 + n)P.$ 
    Consider now the monomorphism
    \[\mu_{n-1, 1} \otimes \mu_{n-1, 1} : V_{2(n-1)}\otimes V_{2(n-1)} \to V_{2n}\otimes V_{2n}.\]
    
    Recall that $A$ is a Frobenius algebra, and hence any projective summand of $V_{2(n-1)}\otimes V_{2(n-1)}$ must embed into a projective summand of $V_{2n}\otimes V_{2n}$. Thus $\mu_{n-1, 1}\otimes \mu_{n-1, 1}$ is isomorphic to a direct sum of monomorphisms $V_{2(n-1)} \hookrightarrow V$ and $P \hookrightarrow (n^2-n)P,$
    for $V = V_{2(n-1)}\oplus V_{2(n+1)}$. 
    By Proposition \ref{classifyingMonos}, we now have that $\mu_{n-1, 1}\otimes \mu_{n-1, 1}$ is isomorphic to $\mu_{n-1, 0} \oplus \mu_{n-1, 2} \oplus \iota$ for the inclusion $\iota : (n^2 - 3n + 2)P \hookrightarrow (n^2 - n)P.$ The cokernel of $\mu_{n-1, 1}\otimes \mu_{n-1, 1}$ must then be isomorphic to the module $V' = V_2 \oplus (2n-2)P.$

    Consider now the exact sequence of $G$-representations
    \[0 \to V_{2(n-1)} \otimes V_{2(n-1)} \xrightarrow{\mu_{n-1, 1}\otimes \mu_{n-1, 1}} V_{2n} \otimes V_{2n} \to V_2 \oplus (2n-2)P \to 0.\]
    Restricting, we have an exact sequence of $C$-representations
    \[0 \to V_{2(n-1)}\downarrow_{C}^G \otimes V_{2(n-1)}\downarrow_{C}^G \to V_{2n}\downarrow_{C}^G \otimes V_{2n}\downarrow_{C}^G \to V_2\downarrow_{C}^G \oplus (2n-2)P\downarrow_{C}^G \to 0.\]
    This is a contradiction: for any $\ell$, $V_{2\ell}\downarrow_{C}^G = 2J_1 \oplus (\ell - 1)J_2,$ and $P\downarrow_{C}^G \cong 2J_2$. Hence
    \begin{flalign*}
        V_{2(n-1)}\downarrow_{C}^G \otimes V_{2(n-1)}\downarrow_{C}^G &\cong 4J_1 \oplus (2(n-1)^2 - 1) J_2,\\
        V_{2n}\downarrow_{C}^G \otimes V_{2n}\downarrow_{C}^G &\cong 4J_1 \oplus (2n^2-1)J_2,\\
        V_2\downarrow_C^G \oplus (2n-1)P\downarrow_C^G &\cong 2J_1 \oplus (4n-4)J_2,
    \end{flalign*}
    and no such exact sequence exists. 
    By Lemma \ref{BasevLinearFunctionalLemma}, since $c_{n+1, n+1, \ell}(\fp) \neq \delta_{n,\ell} + \delta_{n+2, \ell},$ we have instead that 
    \[V_{2n}(\fp)\otimes V_{2n}(\fp) \cong 2V_{2(n)}(\fp) \oplus (n^2 - n)P.\]
\end{proof}

\begin{cor}
    Let $G \in \{\widetilde G, G_1, G_2, G_3\}$. Then $\widetilde G$ is in noble correspondence with $G$. 
\end{cor}
\begin{proof}
    Let $\fp$ be noble for $G$ (it's already noble for $\widetilde G$), and let $\otimes, \widetilde \otimes$ be the products for $G, \widetilde G$ respectively. It follows from duality that $P \otimes V \cong P\widetilde \otimes V \cong nP$ for any $A$-module $V$ of dimension $n$. By Lemma \ref{BasevLinearFunctionalLemma} and Theorem \ref{mainSqTheorem}, for any $n \le m,$ we have
    \[V_{2n}(\fp) \otimes V_{2m}(\fp) \cong 2V_{2n}(\fp) \oplus (mn - n) P \cong V_{2n}(\fp) \widetilde\otimes V_{2n}(\fp). \]
    This accounts for every pair of indecomposable summands of arbitrary modules $M, N$, with $\supp_A(M) = \supp_A(N) = \{\fp\}$ and therefore for any such $M, N$ we have
    \[M\otimes N \cong M\widetilde \otimes N.\]
\end{proof}
\begin{cor}
    Let $\widetilde G$ be the infinitesimal group scheme for an abelian Lie algebra of tame representation type. Then $\widetilde G$ has Property PA. 
\end{cor}
\begin{proof}
    It was argued in the beginning of this section why the previous corollary is sufficient to show that $\widetilde G = \GG_{a(1)}^2,$ in characteristic $2$, has Property PA. Now apply Theorem \ref{RepTypeClassification} and Proposition \ref{PAnullcone}. 
\end{proof}
\printbibliography
\end{document}